\documentclass[12pt,a4paper]{article}
\usepackage{mathrsfs}
\usepackage{epsfig, graphicx}
\usepackage{latexsym,amsfonts,amsbsy,amssymb}
\usepackage{amsmath,amsthm}
\usepackage{color}
\usepackage[colorlinks, citecolor=blue]{hyperref}
\makeatletter 

\@addtoreset{equation}{section}
\makeatother 
\allowdisplaybreaks 

\def\cG{\mathcal{G}}

\def\cO{\mathcal{O}}

\def\cT{\mathcal{T}}

\def\R{\mathbb{R}}

\def\w\eta{\widetilde{\eta}}

\newcommand{\bc}{\begin{center}}
\newcommand{\ec}{\end{center}}
\newcommand{\be}{\begin{eqnarray}}
\newcommand{\ee}{\end{eqnarray}}

\newcommand{\ben}{\begin{eqnarray*}}
\newcommand{\een}{\end{eqnarray*}}

\newtheorem{theorem}{Theorem}[section]
\newtheorem{lemma}{Lemma}[section]

\newtheorem{algorithm}{Algorithm}[section]

\begin{document}
\title{A Multigrid Method for the Ground State Solution of Bose-Einstein Condensates
Based on Newton Iteration\footnote{This work was
supported in part by National Science Foundations of China
(NSFC 91330202, 11371026, 11001259, 11031006, 2011CB309703), the National
Center for Mathematics and Interdisciplinary Science
the national Center for Mathematics and Interdisciplinary Science, CAS.}
}
\author{
Hehu Xie\footnote{LSEC, ICMSEC,
Academy of Mathematics and Systems Science, Chinese Academy of
Sciences, Beijing 100190, China (hhxie@lsec.cc.ac.cn)},\ \
Fei Xu\footnote{LSEC, ICMSEC,
Academy of Mathematics and Systems Science, Chinese Academy of
Sciences, Beijing 100190, China (xufei@lsec.cc.ac.cn)}  \ \ and\ \
Meiling Yue\footnote{LSEC, ICMSEC,
Academy of Mathematics and Systems Science, Chinese Academy of
Sciences, Beijing 100190, China (yuemeiling@lsec.cc.ac.cn)}
}
\date{}
\maketitle
\begin{abstract}
In this paper, a new kind of multigrid method is proposed for the ground state solution
 of Bose-Einstein condensates based on Newton iteration method.
Instead of treating eigenvalue $\lambda$ and eigenvector $u$ respectively,
we regard the eigenpair $(\lambda, u)$
as one element in the composite space $\R \times H_0^1(\Omega)$ and then
 Newton iteration method is adopted for the nonlinear problem.
Thus in this multigrid scheme, we only need to solve a linear discrete
boundary value problem in every refined space, which can improve the overall
efficiency for the simulation  of Bose-Einstein condensations.
\end{abstract}

\section{Introduction}
A Bose-Einstein condensate (BEC) is a state of a dilute gas of bosons cooled
to temperature very close to absolute zero. Under such condition, a large fraction
 of bosons will occupy the lowest quantum state,
at which point, macroscopic quantum becomes apparent. BEC was first proposed by A.
Einstein who
generalized a work of S. N. Bose on the quantum statistics for photons \cite{Bose} to a
gas of non-interacting bosons \cite{Einstein1,Einstein2}.
Then Gross-Pitaevskii theory was developed by  Gross \cite{G} and Pitaevskii \cite{P}
independently in 1960s to describe the dynamics of a BEC \cite{PitaevskiiStringari}.
Since the first experimental
observation of BEC in 1995, much attention has been paid to the Gross-Pitaevskii equation (GPE).

In the past decades, there have existed many papers discussing the numerical methods for the
time-dependent GPEs and time-independent GPEs. Please refer to \cite{Adhikari,AdhikariMuruganandam,AnglinKetterle,BaoCai,BaoDu}
and the papers cited therein. Especially, in \cite{CancesChakirMaday,ZhouBEC}, the convergence
and the priori error estimates of the finite element method for GPEs have been proved, which will be used later in this paper.

Solving such kind of nonlinear eigenvalue problem is an important but difficult problem in science and engineering computation.
As is known to us all, the multigrid method provides an optimal
complexity algorithm to solve discrete boundary value problems.
The aim of this paper is to propose a multigrid scheme for GPEs based on Newton iteration method. More precisely, GPE is regarded as a
nonlinear problem in the composite space $\R\times H_0^1(\Omega)$ and then Newton iteration is adopted to derive a linearized
boundary value problem. Thus, we just need to solve a linear problem with finite element method in every refined space.
With this multigrid scheme, solving GPE problem will not be more difficult than solving the corresponding boundary value problem.
Besides, the convergence rate and computational work of this method are also analyzed in this paper.

An outline of the paper goes as follows. In Section 2, we introduce the finite element
method and corresponding convergence estimates for the ground state solution of BEC, i.e.
non-dimensionalized GPE.
A Newton iteration method for GPE is presented in Section 3.
In Section 4, we propose a type of multigrid algorithm for
GPE based on Newton iteration method. Section 5 is devoted to
estimating the computational work of the multigrid method designed in Section 4.
Two numerical examples are presented in Section 6 to validate the theoretical analysis.
Finally, some concluding remarks are given in the last section.

\section{Finite element method for  Gross-Pitaevskii equation}
This section is devoted to introducing some notation and the finite element method for GPE problem.
The letter $C$ (with or without subscripts) denotes a generic
positive constant which may be different at its different occurrences.
For convenience, the symbols $\lesssim$, $\gtrsim$ and $\approx$
will be used in this paper to denote $x_1\leq C_1y_1$, $x_2 \geq c_2y_2$
and $c_3x_3\leq y_3\leq C_3x_3$ for some constants $C_1, c_2, c_3$, $C_3$
that are independent of mesh sizes (see, e.g., \cite{Xu}).
We shall use the standard notation for Sobolev spaces $W^{s,p}(\Omega)$ and their
associated norms $\|\cdot\|_{s,p,\Omega}$ and seminorms $|\cdot|_{s,p,\Omega}$
(see, e.g., \cite{Adams}). For $p=2$, we denote
$H^s(\Omega)=W^{s,2}(\Omega)$, $H_0^1(\Omega)=\{v\in H^1(\Omega):\ v|_{\partial\Omega}=0\}$,
where $v|_{\partial\Omega}=0$ is in the sense of trace and
$\|\cdot\|_{s,\Omega}=\|\cdot\|_{s,2,\Omega}$. In this paper, we set $V=H_0^1(\Omega)$
and use $\|\cdot\|_s$ to denote $\|\cdot\|_{s,\Omega}$ for simplicity.

It is known that the wave function $\psi$ of a sufficiently dilute condensate,
in the presence of an external potential $\widetilde{W,}$  satisfies the following GPE
\begin{eqnarray}\label{GPE}
\left(-\frac{\hbar^2}{2m}\Delta + \widetilde{W} + \frac{4\pi\hbar^2aN}{m}|\psi|^2\right)\psi &=&\mu\psi,
\end{eqnarray}
where $\mu$ is the chemical potential and $N$ is the number of atoms
in the condensate, $4\pi\hbar^2a/m$ represents the effective two-body interaction, $\hbar$ is the
Plank constant, $a$ is the scattering length (positive for repulsive interaction and negative for
attractive interaction) and $m$ is the particle mass. In this paper, we assume the external potential
 $\widetilde{W}(x)$ is measurable, locally bounded and tends to infinity as $|x|\rightarrow\infty$
in the sense that
\begin{eqnarray*}
\inf_{|x|\geq r}\widetilde{W}(x)\rightarrow \infty\ \ \ \ {\rm as}\ r\rightarrow \infty.
\end{eqnarray*}
Then the wave function $\psi$ must vanish exponentially fast as $|x|\rightarrow \infty$.
Furthermore, (\ref{GPE}) can be written as
\begin{eqnarray}\label{GPE_Simple}
\left(-\Delta +\frac{2m}{\hbar^2}\widetilde{W}+8\pi aN|\psi|^2\right)\psi&=&\frac{2m\mu}{\hbar^2}\psi.
\end{eqnarray}
Hence in this paper, we are concerned with the  smallest eigenpair for
the following non-dimensionalized GPE problem:
\begin{equation}\label{GPEsymply2}
\left\{
\begin{array}{rcl}
-{\Delta} u + Wu + \zeta |u|^2u &=& \lambda u,\ \ \  {\rm in}\  \Omega,\\
u &=& 0,\ \ \ \ \  {\rm on}\ \partial \Omega,\\
\int_\Omega|u|^2{d\Omega} &=& 1,
\end{array}
\right.
\end{equation}
where $\Omega \subset \mathbb{R}^d$  $( d = 1, 2, 3)$ denotes the computing
domain which has the cone property \cite{Adams}, $\zeta$ is some positive constant and
$W(x) = \gamma_1 x^2_1 +\cdots +\gamma_d x^2_d \geq 0$
with  $\gamma_1,\cdots , \gamma_d > 0$ \cite{BaoTang,ZhouBEC}.

For the aim of finite element
discretization, the corresponding weak form for (\ref{GPEsymply2}) can be described as follows:
Find $(\lambda,u)\in \R\times V$ such that $b(u,u) = 1$ and
\begin{equation}\label{GPEweakform}
a(u,v) = \lambda b(u,v),\ \ \  \forall v \in V,
\end{equation}
where
\[
a(u,v) = \int_\Omega\big( \nabla u\cdot\nabla v+Wuv+\zeta |u|^2uv\big)d\Omega,
\ \ \ \ b(u,v) = \int_\Omega uvd\Omega.
\]
We also introduce the linearized form $a'(u;v,w)$ by
\begin{eqnarray}
a'(u;v,w) = \int_\Omega\big(\nabla v\cdot\nabla w + Wvw
+ 3\zeta|u|^2vw\big)d\Omega, \quad \forall v,w\in V.
\end{eqnarray}
Here and hereafter in this paper, we only consider the smallest eigenvalue and
the corresponding eigenfunction of the problem (\ref{GPEweakform}).
For GPE problem, we can find the following estimates from \cite{CancesChakirMaday}.
\begin{lemma}
There exist positive constants $M,\  C_L$ and $C_U$ such that for all $v\in H_0^1(\Omega)$,
\begin{eqnarray}
0\leq ( \nabla v,\nabla v)+(Wv+\zeta |u|^2v,v)-\lambda(v,v) \leq M\|v\|_1^2
\end{eqnarray}
and
\begin{eqnarray}\label{coercive a'}
C_L \|v\|_1^2\leq a'(u;v,v)-\lambda(v,v) \leq C_U\|v\|_1^2.
\end{eqnarray}
\end{lemma}

Now, let us define the finite element method \cite{BrennerScott,Ciarlet}
for the problem (\ref{GPEweakform}). First we
generate a shape-regular decomposition of the computing domain
$\Omega \subset \mathbb{R}^d$ and the diameter of a cell $K \in \mathcal{T}_h$ is
denoted by $h_K$. The mesh diameter $h$ describes the maximum diameter of all cells
$K \in \mathcal{T}_h$. Based on the mesh $\mathcal{T}_h$, we construct the
linear finite element space denoted by $V_h\subset V$.
We assume that $V_h\subset V$ satisfies the following assumption:
\begin{equation}
\lim_{h\rightarrow 0}\inf_{v \in V_h} \|w - v\|_1 = 0,\ \ \ \forall w\in V.
\end{equation}

The standard finite element method for (\ref{GPEweakform}) is to solve the following eigenvalue problem:
Find $(\lambda_h,u_h)\in \R\times V_h$ such that $b(u_h,u_h) = 1$ and
\begin{equation}\label{GPEfem}
a(u_h,v_h) = \lambda_h b(u_h,v_h),\ \ \   \forall v_h \in V_h.
\end{equation}
Then we define
\begin{equation}\label{delta}
\delta_h(u) := \inf_{v_h\in V_h}\|u - v_h\|_1.
\end{equation}


The convergence estimates of the finite element method for (\ref{GPEweakform}) are
presented in the following lemma which can be found in \cite{CancesChakirMaday,ZhouBEC}.
\begin{lemma}\label{lemma:Maday}
(\cite[Theorem 1]{CancesChakirMaday})
There exists $h_0 > 0$, such that for all $0 < h < h_0$, the smallest eigenpair approximation
 $(\lambda_h,u_h)$ of (\ref{GPEfem}) has the following error estimates
\begin{eqnarray}
\|u - u_h\|_1 &\lesssim& \delta_h(u),\\
\|u - u_h\|_0 &\lesssim& \eta_a(V_h)\|u -u_h\|_1 \lesssim \eta_a(V_h)\delta_h(u),\\
|\lambda - \lambda_h| &\lesssim& \|u - u_h\|^2_1
+ \|u - u_h\|_0\lesssim \eta_a(V_h)\delta_h(u),
\end{eqnarray}
where $\eta_a(V_h)$ is defined as follows
\begin{eqnarray}\label{eta_a_h}
\eta_a(V_h)=\|u-u_h\|_1+ \sup_{f\in L^2(\Omega),\|f\|_0=1}\inf_{v_h\in V_h}\|Tf-v_h\|_1
\end{eqnarray}
with the operator $T$ being defined as follows:
 Find $Tf\in u^{\perp}$ such that
\vskip-0.7cm
\begin{eqnarray*}
a'(u;Tf,v)- (\lambda(Tf),v)=(f,v),\ \ \ \ \forall v\in u^{\perp},
\end{eqnarray*}
and $u^{\perp}=\big\{v\in H_0^1(\Omega) | \int_{\Omega}uvd\Omega=0\big\}$.
\end{lemma}
\section{Newton iteration method for Gross-Pitaevskii equation}
In this section, Newton iteration method is introduced to solve the GPE
problem in a composite space defined as follows:

Denote the space $\R\times H_0^1(\Omega)$ by $X$ and $\R\times H^{-1}(\Omega)$
by $X^*$ with the norm
\begin{eqnarray*}
\|(\gamma,w)\|_X=|\gamma |+\|w\|_1 \quad
\text{and} \quad
\|(\gamma,w)\|_0=|\gamma |+\|w\|_0, \quad \forall  (\gamma,w)\in X.
\end{eqnarray*}
And the corresponding finite element space $\R\times V_h$ is denoted by $X_h$.

For any $(\gamma,w),\ (\mu,v)\in X$, we define a nonlinear operator $\cG: X \rightarrow X^*$ as follows
\begin{eqnarray}\label{definition of g}
 \langle\cG(\gamma,w),(\mu,v)\rangle &=& ( \nabla w,\nabla v)+(Ww+\zeta |w|^2w-\gamma w,v)\nonumber\\
 &&\ \ \  +\frac{1}{2}\mu\left(1-\int_{\Omega}w^2d\Omega\right).
\end{eqnarray}
Since we request $\|u\|_0^2=1$, (\ref{GPEweakform}) can be rewritten as: Find $(\lambda,u)\in X$ such that
\begin{eqnarray}\label{rewrittem g}
\langle\cG(\lambda,u),(\mu,v)\rangle=0, \quad \forall (\mu,v) \in X.
\end{eqnarray}
The Fr\'{e}chet derivation of $\cG$ at $(\lambda,u)$ is given by
\begin{eqnarray}\label{definition of g'}
\langle\cG'(\lambda,u)(\gamma,w),(\mu,v)\rangle&=&(\nabla w,\nabla v)+((W+3\zeta u^2-\lambda)w,v)\nonumber\\
&&\ \ \ -\gamma(u,v)-\mu(u,w)\nonumber\\
&=&a'(u;w,v)-\lambda (w,v)-\gamma(u,v)-\mu(u,w).
\end{eqnarray}
Assume we have an initial eigenpair approximation $(\lambda^{'},u^{'})$ on the finite element space $X_h$,
Newton iteration method for GPE
is defined as follows to get a better eigenpair approximation $(\lambda^{''},u^{''})\in X_h$:
\begin{eqnarray}\label{Newton Iteration}
\langle \cG'(\lambda^{'},u^{'})(\lambda^{''}-\lambda^{'},u^{''}-u^{'}),(\mu,v)\rangle
=-\langle\cG(\lambda^{'},u^{'}),(\mu,v)\rangle, \quad \forall (\mu,v)\in X_h.
\end{eqnarray}
From (\ref{definition of g}) and (\ref{definition of g'}), (\ref{Newton Iteration})
can be rewritten as follows: Find$(\lambda^{''},u^{''})\in X_h$, such that
\begin{eqnarray}\label{newton step}
\begin{cases}
a(u^{'};u^{''},v)-\lambda^{''}(u^{'},v)
=(2\zeta (u^{'})^3-\lambda^{'}u^{'},v),\quad \forall v\in V_h,\\
-\mu (u^{'},u^{''})=-\mu/2-\mu(u^{'},u^{'})/2,\qquad\qquad\quad\ \ \forall \mu\in \R
\end{cases}
\end{eqnarray}
with $a(u^{'};u^{''},v)=a'(u^{'};u^{''},v)-\lambda^{'}(u^{''},v)$.

The isomorphism property of $\cG'$ is analyzed in the following theorem.

\begin{theorem}
If the mesh size $h$ is sufficiently small, then for the linearized operator
 $\cG'$ presented in (\ref{definition of g'}), we have
\begin{eqnarray}\label{coercive of g1}
\|(\gamma,w)\|_X \lesssim \mathop {\sup}
\limits_{(\mu,v)\in X}\dfrac{\langle\cG'(\lambda,u)(\gamma,w),(\mu,v)\rangle}{\|(\mu,v)\|_X},
\quad \forall (\gamma,w)\in X
\end{eqnarray}
and
\begin{eqnarray}\label{coercive of g2}
  \|(\gamma,w)\|_X \lesssim \mathop {\sup}
  \limits_{(\mu,v)\in X_h}\dfrac{\langle\cG'(\lambda,u)(\gamma,w),(\mu,v)\rangle}{\|(\mu,v)\|_X},
  \quad \forall (\gamma,w)\in X_h.
\end{eqnarray}
For any $(\hat{\lambda},\hat{u})\in X$ such that
$\|(\hat{\lambda}-\lambda,\hat{u}-u)\|_X$ is small enough, there holds
\begin{eqnarray}\label{coercive of g3}
  \|(\gamma,w)\|_X \lesssim \mathop {\sup} \limits_{(\mu,v)\in X_h}\dfrac{\langle\cG'(\hat{\lambda},\hat{u})(\gamma,w),(\mu,v)\rangle}{\|(\mu,v)\|_X},
  \quad \forall (\gamma,w)\in X_h.
\end{eqnarray}
\end{theorem}
\begin{proof}
For the first estimate (\ref{coercive of g1}), we just need to prove that the equation
\begin{eqnarray}\label{linear equation}
\cG'(\lambda,u)(\gamma,w)=(\tau,f)
\end{eqnarray}
is uniquely solvable in $X$ for any $(\tau,f)\in X^*$.
From (\ref{definition of g'}), we obtain that (\ref{linear equation}) can be rewritten as
\begin{eqnarray*}
\begin{cases}
a'(u;w,v)-\lambda(w,v)+b_u(\gamma,v)=(f,v),   \quad \forall v\in V,\\
b_u(\mu,w)=\mu\tau,  \qquad \qquad \qquad \qquad\qquad ~~~ \quad\forall \mu \in \R,
\end{cases}
\end{eqnarray*}
where $b_u(\mu,v)=-\mu(u,v)$.

For this saddle problem, the solvable condition is (Theorem 1.1 in \cite{Fortin},
 \uppercase\expandafter{\romannumeral2}):

Firstly, the following variational problem
\begin{eqnarray}\label{unique of a}
a'(u;w,v)-\lambda (w,v)=(f,v), \quad \forall v\in V_0
\end{eqnarray}
is uniquely solvable for any $f\in H^{-1}(\Omega)$ and $V_0:=\{v\in V: b_u(\mu,v)=0, \forall \mu \in \R\}$.

Secondly, $b_u(\cdot,\cdot)$ satisfies the inf-sup condition
\begin{eqnarray}\label{infsup}
\inf_{\mu\in \R}\sup_{v\in V}\dfrac{b_u(\mu,v)}{\|v\|_1|\mu|}\geq k_b
\end{eqnarray}
for some constant $k_b>0$.

The well-poseness of (\ref{unique of a}) can be derived from (\ref{coercive a'}) directly.

For the second condition (\ref{infsup}), take $v=-\mu u$. Since $\|u\|_0=1$, there holds
\begin{eqnarray*}
\inf_{\mu\in \R}\sup_{v\in V}\dfrac{b_u(\mu,v)}{\|v\|_1|\mu|} \geq
\inf_{\mu\in \R}\dfrac{\mu^2(u,u)}{\|u\|_1|\mu|^2}=\frac{1}{\|u\|_1}=: k_b.
\end{eqnarray*}
This completes the proof of (\ref{coercive of g1}).

From (\ref{coercive a'}), we can define a project operator $P_h:V \rightarrow V_h$ by
\begin{eqnarray}
a'(u;w,v-P_hv)-\lambda(w,v-P_hv)=0, \quad \forall w\in V_h,\ v\in V.
\end{eqnarray}
There apparently holds
\begin{eqnarray}\label{projection}
\|P_hv\|_1 \lesssim \|v\|_1, \quad\forall v \in V.
\end{eqnarray}
From the Aubin-Nitsche lemma, we have
\begin{eqnarray}\label{L2_Estimate}
\|v-P_hv\|_0\lesssim \eta_a(V_h)\|v\|_1, \quad\forall v \in V.
\end{eqnarray}
So for any $(\gamma,w)\in X_h$, from (\ref{L2_Estimate}), the following estimates hold
\begin{eqnarray}\label{isomorphism}
 \|(\gamma,w)\|_X &\lesssim&  \mathop {\sup} \limits_{(\mu,v)\in X}\dfrac{\langle\cG'(\lambda,u)(\gamma,w),(\mu,v)\rangle}{\|(\mu,v)\|_X}\nonumber\\
 &=&\mathop {\sup} \limits_{(\mu,v)\in X}\dfrac{\langle\cG'(\lambda,u)(\gamma,w),(\mu,P_hv)\rangle
 +\langle\cG'(\lambda,u)(\gamma,w),(0,v-P_hv)\rangle}{\|(\mu,v)\|_X}\nonumber\\
 &=&\mathop {\sup} \limits_{(\mu,v)\in X}\dfrac{\langle\cG'(\lambda,u)(\gamma,w),(\mu,P_hv)\rangle
 -\gamma(u,v-P_hv)}{\|(\mu,v)\|_X}\nonumber\\
 &\lesssim&\mathop {\sup} \limits_{(\mu,v)\in X}\dfrac{\langle\cG'(\lambda,u)(\gamma,w),(\mu,P_hv)\rangle+\gamma \|u\|_0\|v-P_hv\|_0}{\|(\mu,v)\|_X}\nonumber\\
 &\lesssim&\mathop {\sup} \limits_{(\mu,v)\in X}\dfrac{\langle\cG'(\lambda,u)(\gamma,w),(\mu,P_hv)\rangle+\gamma \eta_a(V_h)\|u\|_0\|v\|_1}{\|(\mu,v)\|_X}\nonumber\\
 &\lesssim&\mathop {\sup} \limits_{(\mu,v)\in X}\dfrac{\langle\cG'(\lambda,u)(\gamma,w),(\mu,P_hv)\rangle}{\|(\mu,v)\|_X}+\eta_a(V_h)\|(\gamma,w)\|_X.
 \end{eqnarray}
Combing (\ref{projection}) and (\ref{isomorphism}) leads to
\begin{eqnarray*}
 \|(\gamma,w)\|_X &\lesssim&\mathop {\sup} \limits_{(\mu,v)\in X}\dfrac{\langle\cG'(\lambda,u)(\gamma,w),(\mu,P_hv)\rangle}{\|(\mu,v)\|_X}\\
 &\lesssim&\mathop {\sup} \limits_{(\mu,v)\in X}
 \dfrac{\langle\cG'(\lambda,u)(\gamma,w),(\mu,P_hv)\rangle}{\|(\mu,P_hv)\|_X}\\
 &\lesssim&\mathop {\sup} \limits_{(\mu,v)\in X_h}
 \dfrac{\langle\cG'(\lambda,u)(\gamma,w),(\mu,v)\rangle}{\|(\mu,v)\|_X}.
 \end{eqnarray*}
Then we get the desired conclusion (\ref{coercive of g2}).

For the last inequality (\ref{coercive of g3}), we assume there exists
a sufficiently small constant $\varepsilon$ such that
$\|(\hat{\lambda}-\lambda,\hat{u}-u)\|_X\leq \varepsilon$. Then for any $(\gamma,w)\in X_h$
\begin{eqnarray*}
 \|(\gamma,w)\|_X &\lesssim&  \mathop {\sup} \limits_{(\mu,v)\in X_h}\dfrac{\langle\cG'(\lambda,u)(\gamma,w),(\mu,v)\rangle}{\|(\mu,v)\|_X}\\
 &\lesssim&\mathop {\sup} \limits_{(\mu,v)\in X_h}
 \dfrac{\langle\cG'(\hat{\lambda},\hat{u})(\gamma,w),(\mu,v)\rangle+
 \varepsilon \|(\gamma,w)\|_X\|(\mu,v)\|_X }{\|(\mu,v)\|_X}.
 \end{eqnarray*}
The desired result $(\ref{coercive of g3})$ then easily follows if $\varepsilon$ is sufficiently small.
\end{proof}
Applying Newton iteration method to GPE leads to a linearized problem,
the corresponding residual estimate can be derived from the following theorem.
\begin{theorem}\label{residual estiamte}
For the nonlinear operator $\cG$ and any $(\mu_h,v_h), (\mu,v)\in X$, we have
\begin{eqnarray}\label{error of Newton Iteration}
\langle \cG(\mu_h,v_h),(\sigma,\eta)\rangle&=&\langle\cG(\mu,v),(\sigma,\eta)\rangle
+\langle \cG'(\mu,v)(\mu_h-\mu,v_h-v),(\sigma,\eta)\rangle \nonumber\\
&&+R\big((\mu,v),(\mu_h,v_h),(\sigma,\eta)\big),\quad\quad\forall (\sigma,\eta)\in X
\end{eqnarray}
with $R\big((\mu,v),(\mu_h,v_h),(\sigma,\eta)\big)$ being the residual which can be estimated as follows:
\begin{eqnarray*}
| R\big((\mu,v),(\mu_h,v_h),(\sigma,\eta)\big)|
\lesssim \|(\mu-\mu_h,v-v_h)\|_X\|(\mu-\mu_h,v-v_h)\|_0\|(\sigma,\eta)\|_X.
\end{eqnarray*}
\end{theorem}
\begin{proof}
Define
\begin{eqnarray}
\varphi(t)=\langle \cG((\mu,v)+t(\mu_h-\mu,v_h-v)),(\sigma,\eta)\rangle.
\end{eqnarray}
Then the derivative of $\varphi$ with respect to $t$ can be derived trivially.
\begin{eqnarray*}
\varphi'(t)&=&(\nabla (v_h-v),\nabla \eta)+(W(v_h-v),\eta)+3 \big(\zeta(v+t(v_h-v))^2(v_h-v),\eta\big)\\
&&-(\mu_h-\mu)(v+t(v_h-v),\eta)-(\mu+t(\mu_h-\mu))(v_h-v,\eta)\\
&&-\sigma (v+t(v_h-v),v_h-v)\\
&=&\langle \cG'((\mu,v)+t(\mu_h-\mu,v_h-v))(\mu_h-\mu,v_h-v),(\sigma,\eta)\rangle
\end{eqnarray*}
and
\begin{eqnarray}\label{eta''}
\varphi''(t)&=&-2(\mu_h-\mu)(v_h-v,\eta)-\sigma (v_h-v,v_h-v) \nonumber\\
&& \ \ +6(\zeta(v+t(v_h-v))(v_h-v)^2,\eta).
\end{eqnarray}
Denote $\xi = v+t(v_h-v)$ and from the imbedding theorem, we have
$$\|\xi\|_{0,6}\lesssim \|\xi\|_1 \lesssim \|v\|_1+\|v_h\|_1.$$
For the last term of (\ref{eta''}),
\begin{eqnarray}
|(\xi(v_h-v)^2,\eta)| &\lesssim& \int_{\Omega}|\xi|(v_h-v)^2|\eta| dx \nonumber\\
&\lesssim& \|\xi\|_{0,6}\|v_h-v\|_{0}\|v_h-v\|_{0,6}\|\eta\|_{0,6} \nonumber\\
&\lesssim& \|\xi\|_1\|v_h-v\|_1\|v_h-v\|_0\|\eta\|_1.
\end{eqnarray}
Thus, (\ref{error of Newton Iteration}) can be derived from the following Taylor expansion
\begin{eqnarray}\label{talor expansion}
\varphi(1)=\varphi(0)+\varphi'(0)+\int_{0}^{1}\varphi''(t)(1-t)dt.
\end{eqnarray}
Due to (\ref{eta''})-(\ref{talor expansion}), the residual $R$ satisfies
\begin{eqnarray*}
| R\big((\mu,v),(\mu_h,v_h),(\sigma,\eta)\big)|
\lesssim \|(\mu-\mu_h,v-v_h)\|_X\|(\mu-\mu_h,v-v_h)\|_0\|(\sigma,\eta)\|_X.
\end{eqnarray*}
Then we complete the proof.
\end{proof}


\section{Multigrid algorithm based on Newton iteration method}
In this section, we propose a multigrid scheme based on Newton iteration method.
In this algorithm, we only need to solve a linearized mixed variational problem
on each refined finite element space.

\subsection{One Newton iteration step}
In order to design the multigrid method, we first introduce an one Newton iteration
step in this subsection. Assume we have obtained an eigenpair approximation
$(\lambda^{h_k},u^{h_k})\in \mathcal{R}\times V_{h_k}$, a type of iteration
step will be introduced to derive an eigenpair
$(\lambda^{h_{k+1}},u^{h_{k+1}})\in \mathcal{R}\times V_{h_{k+1}}$ with a  better accuracy.
In this paper, we denote by $(\lambda_{h_k},u_{h_k})$ the standard finite element solution of (\ref{GPEweakform}).
\begin{algorithm}\label{One Correction Step}
One Newton Iteration Step
\begin{enumerate}
\item Define the linearized mixed variational equation on the finite element space $X_{h_{k+1}}$ as follows:

Find $(\hat{\lambda}^{h_{k+1}},\hat{u}^{h_{k+1}})\in X_{h_{k+1}}$
such that for all $(\mu,v_{h_{k+1}})\in X_{h_{k+1}}$
\begin{eqnarray}\label{one correction bvp}
\begin{cases}
a(u^{h_{k}};\hat{u}^{h_{k+1}},v_{h_{k+1}})-\hat{\lambda}^{h_{k+1}}(u^{h_{k}},v_{h_{k+1}})
=(2\zeta(u^{h_{k}})^3-\lambda^{h_{k}}u^{h_{k}},v_{h_{k+1}}),  \\
-\mu (\hat{u}^{h_{k+1}}, u^{h_{k}})=-\mu/2-\mu(u^{h_{k}},u^{h_{k}})/2,
\end{cases}
\end{eqnarray}
where
$a(u^{h_{k}};\hat{u}^{h_{k+1}},v_{h_{k+1}})=a'(u^{h_{k}};\hat{u}^{h_{k+1}},
v_{h_{k+1}})-\lambda^{h_{k}}(\hat{u}^{h_{k+1}},v_{h_{k+1}})$.
\item Solve equation $(\ref{one correction bvp})$ to obtain an eigenpair
approximation $(\lambda^{h_{k+1}},u^{h_{k+1}})$
satisfying $\|(\lambda^{h_{k+1}}-\hat{\lambda}^{h_{k+1}},u^{h_{k+1}}-\hat{u}^{h_{k+1}})\|_X\lesssim \eta_a(V_{h_{k+1}})\delta_{h_{k+1}}(u) $.
\end{enumerate}
In order to simplify the notation and summarize the above two steps, we define
\begin{eqnarray*}
(\lambda^{h_{k+1}},u^{h_{k+1}})=Newton_{-}Iteration(\lambda^{h_k}, u^{h_k},V_{h_{k+1}}).
\end{eqnarray*}
\end{algorithm}

\begin{theorem}\label{One Correction Step estimate}
After implementing Algorithm $\ref{One Correction Step}$, the resultant
eigenpair approximation $(\lambda^{h_{k+1}},u^{h_{k+1}})$ has the following error estimate
\begin{eqnarray}\label{estimate}
&&\|(\lambda^{h_{k+1}}-\lambda_{h_{k+1}},u^{h_{k+1}}-u_{h_{k+1}})\|_X\nonumber\\
&\lesssim& \eta_a(V_{h_{k+1}})\delta_{h_{k+1}}(u)+\delta_{h_{k}}(u)
\|(\lambda_{h_{k}}-\lambda^{h_{k}},u_{h_{k}}-u^{h_{k}})\|_X\nonumber\\
&&+\|(\lambda_{h_{k}}-\lambda^{h_{k}},u_{h_{k}}-u^{h_{k}})\|_X
\|(\lambda_{h_{k}}-\lambda^{h_{k}},u_{h_{k}}-u^{h_{k}})\|_0.
\end{eqnarray}
\end{theorem}
\begin{proof}
For the standard finite element solution $(\lambda_{h_{{k+1}}},u_{h_{{k+1}}})$, we have
\begin{eqnarray}
\langle\cG(\lambda_{h_{{k+1}}},u_{h_{{k+1}}}),(\mu,v_{h_{{k+1}}})\rangle=0,
\quad \forall (\mu,v_{h_{{k+1}}}) \in X_{h_{k+1}}.
\end{eqnarray}
Together with Theorem \ref{residual estiamte} and Algorithm \ref{One Correction Step}, there holds
\begin{eqnarray*}
&&\langle \cG'(\lambda^{h_{k}},u^{h_{k}})(\lambda_{h_{k+1}}-\hat{\lambda}^{h_{k+1}},u_{h_{k+1}}
-\hat{u}^{h_{k+1}}),(\mu,v_{h_{k+1}})\rangle \\
&=& \langle \cG(\lambda^{h_{k}},u^{h_{k}}),(\mu,v_{h_{k+1}})\rangle\\
&&+\langle \cG'(\lambda^{h_{k}},u^{h_{k}})(\lambda_{h_{k+1}}
-\lambda^{h_{k}},u_{h_{k+1}}-u^{h_{k}}),(\mu,v_{h_{k+1}})\rangle \\
&=& \langle \cG(\lambda^{h_{k}},u^{h_{k}}),(\mu,v_{h_{k+1}})\rangle
-\langle\cG(\lambda_{h_{k+1}},u_{h_{k+1}}),(\mu,v_{h_{k+1}})\rangle\\
&&+\langle \cG'(\lambda^{h_{k}},u^{h_{k}})(\lambda_{h_{k+1}}
-\lambda^{h_{k}},u_{h_{k+1}}-u^{h_{k}}),(\mu,v_{h_{k+1}})\rangle \\
&=&-R\big((\lambda^{h_{k}},u^{h_{k}}),(\lambda_{h_{k+1}},u_{h_{k+1}}),(\mu,v_{h_{k+1}})\big).
\end{eqnarray*}
Using (\ref{coercive of g3}) and Theorem \ref{residual estiamte}, we derive
\begin{eqnarray}\label{coercive in Newtopn}
&&\|(\lambda_{h_{k+1}}-\hat{\lambda}^{h_{k+1}},u_{h_{k+1}}-\hat{u}^{h_{k+1}})\|_X\nonumber\\
&\lesssim& \|(\lambda_{h_{k+1}}-\lambda^{h_{k}},u_{h_{k+1}}-u^{h_{k}})\|_X\|(\lambda_{h_{k+1}}
-\lambda^{h_{k}},u_{h_{k+1}}-u^{h_{k}})\|_0\nonumber\\
&\lesssim& \eta_a(V_{h_{k}})\delta_{h_{k}}^2(u)+\delta_{h_{k}}(u)\|(\lambda_{h_{k}}
-\lambda^{h_{k}},u_{h_{k}}-u^{h_{k}})\|_X\nonumber\\
&&+\|(\lambda_{h_{k}}-\lambda^{h_{k}},u_{h_{k}}-u^{h_{k}})\|_X\|(\lambda_{h_{k}}
-\lambda^{h_{k}},u_{h_{k}}-u^{h_{k}})\|_0.
\end{eqnarray}
Since
\begin{eqnarray*}
\|(\lambda^{h_{k+1}}-\hat{\lambda}^{h_{k+1}},u^{h_{k+1}}-\hat{u}^{h_{k+1}})\|_X
\lesssim \eta_a(V_{h_{k+1}})\delta_{h_{k+1}}(u),
\end{eqnarray*}
we arrive at
\begin{eqnarray*}
&&\|(\lambda^{h_{k+1}}-\lambda_{h_{k+1}},u^{h_{k+1}}-u_{h_{k+1}})\|_X\\
&\lesssim& \eta_a(V_{h_{k+1}})\delta_{h_{k+1}}(u)+\delta_{h_{k}}(u)\|(\lambda_{h_{k}}
-\lambda^{h_{k}},u_{h_{k}}-u^{h_{k}})\|_X\\
&&+\|(\lambda_{h_{k}}-\lambda^{h_{k}},u_{h_{k}}-u^{h_{k}})\|_X\|(\lambda_{h_{k}}
-\lambda^{h_{k}},u_{h_{k}}-u^{h_{k}})\|_0..
\end{eqnarray*}
This completes the proof.
\end{proof}

\subsection{Multigrid method}
In order to do multigrid iteration, we define a sequence of triangulations
$\cT_{h_k}$ and $\cT_{h_{k+1}}$ is produced from $\cT_{h_k}$ via a regular
refinement (produce $\beta^d$ congruent elements) such that
\begin{eqnarray}\label{mesh_size_recur}
h_k\approx\frac{1}{\beta}h_{k-1},
\end{eqnarray}
where the integer $\beta$ denotes the refinement index and larger than $1$ (always equals $2$).
Based on the mesh sequence, we construct a sequence of linear finite element spaces satisfying
\begin{eqnarray}\label{mesh relationship}
V_{h_{1}}\subset V_{h_{2}}\subset \cdots \subset V_{h_n}\subset V
\end{eqnarray}
and assume the following relations of approximation errors hold
\begin{eqnarray}\label{delta_recur_relation}
\eta_a(V_{h_k})\approx\frac{1}{\beta}\eta_a(V_{h_{k-1}}),\quad
\delta_{h_k}(u)\approx\frac{1}{\beta}\delta_{h_{k-1}}(u),\ \ \ k=1,2,\cdots,n.
\end{eqnarray}
Obviously, the following relationship is also valid
 \begin{eqnarray}
 X_{h_{1}}\subset X_{h_{2}}\subset \cdots \subset X_{h_n}\subset X.
 \end{eqnarray}

The multigrid method based on one Newton iteration step is proposed in the following algorithm.

\begin{algorithm}\label{mg}
Multigrid Algorithm
\begin{enumerate}
\item Construct a series of nested finite element spaces $V_{h_1},V_{h_2},\cdots,V_{h_n}$
such that $(\ref{mesh relationship})$ and $(\ref{delta_recur_relation})$ hold.
\item Solve the GPE on the initial finite element space $X_{h_1}$:
Find $(\lambda^{h_1},u^{h_1})\in X_{h_1}$ such that
\begin{equation*}
(\nabla u^{h_1},\nabla v_{h_1})+(Wu^{h_1},v_{h_1})+(\zeta( u^{h_1})^3,v_{h_1})
=\lambda^{h_1}(u^{h_1},v_{h_1}), \quad \forall v\in V_{h_1}.
\end{equation*}
\item Do $k=1,\cdots,n-1$\\
Obtain a new eigenpair approximation $(\lambda^{h_{k+1}},u^{h_{k+1}})$ by a Newton iteration step
\begin{eqnarray}
(\lambda^{h_{k+1}},u^{h_{k+1}}) = Newton_{-}Iteration(\lambda^{h_k},u^{h_k},V_{h_{k+1}}).
\end{eqnarray}
End Do.
\end{enumerate}
\end{algorithm}
\begin{theorem}\label{mg estimate}
Assume the initial mesh size $h_1$ is sufficiently small, after implementing
Algorithm \ref{mg}, the resultant eigenpair
approximation $(\lambda^{h_n},u^{h_n})$ has the following error estimate
\begin{eqnarray}\label{estimate}
\|(\lambda_{h_n}-\lambda^{h_n},u_{h_n}-u^{h_n})\|_X &\lesssim& \eta_a(V_{h_n})\delta_{h_n}(u).
\end{eqnarray}
\end{theorem}
\begin{proof}
From the second step of Algorithm \ref{mg}, 
we have
\begin{eqnarray}\label{initial mesh}
0=\|(\lambda_{h_1}-\lambda^{h_1},u_{h_1}-u^{h_1})\|_X &\lesssim& \eta_a(V_{h_1})\delta_{h_1}(u).
\end{eqnarray}
Using Theorem \ref{One Correction Step estimate}, we derive
\begin{eqnarray}\label{coercive in Newtopn}
&&\|(\lambda_{h_{n}}-\hat{\lambda}^{h_{n}},u_{h_{n}}-\hat{u}^{h_{n}})\|_X\nonumber\\
&\lesssim& \|(\lambda_{h_{n}}-\lambda^{h_{n-1}},u_{h_{n}}-u^{h_{n-1}})\|_X\|(\lambda_{h_{n}}
-\lambda^{h_{n-1}},u_{h_{n}}-u^{h_{n-1}})\|_0\nonumber\\
&\lesssim& \eta_a(V_{h_{n}})\delta_{h_{n}}(u)+\delta_{h_{n-1}}(u)\|(\lambda_{h_{n-1}}
-\lambda^{h_{n-1}},u_{h_{n-1}}-u^{h_{n-1}})\|_X\nonumber\\
&&\hskip-1cm+\|(\lambda_{h_{n-1}}-\lambda^{h_{n-1}},u_{h_{n-1}}-u^{h_{n-1}})\|_X\|(\lambda_{h_{n-1}
}-\lambda^{h_{n-1}},u_{h_{n-1}}-u^{h_{n-1}})\|_0.
\end{eqnarray}

Inequality  (\ref{initial mesh}) means that (\ref{estimate}) holds for the initial finite element space $X_{h_1}$.
Assume that (\ref{estimate}) is true for the space $V_{h_{n-1}}$, i.e.,
\begin{eqnarray}\label{estimate n-1}
 \|(\lambda_{h_{n-1}}-\lambda^{h_{n-1}},u_{h_{n-1}}-u^{h_{n-1}})\|_X \lesssim \eta_a(V_{h_{n-1}})\delta_{h_{n-1}}(u).
 \end{eqnarray}
Combining (\ref{coercive in Newtopn}) and (\ref{estimate n-1}) leads to
\begin{eqnarray*}
 \|(\lambda_{h_{n}}-\hat{\lambda}^{h_{n}},u_{h_{n}}-\hat{u}^{h_{n}})\|_X \lesssim  \eta_a(V_{h_n})\delta_{h_{n}}(u).
 \end{eqnarray*}
 Since
 \begin{eqnarray*}
 \|(\lambda^{h_{n}}-\hat{\lambda}^{h_{n}},u^{h_{n}}-\hat{u}^{h_{n}})\|_X
\lesssim \eta_a(V_{h_n})\delta_{h_{n}}(u),
 \end{eqnarray*}
 we arrive at
 \begin{eqnarray*}
 \|(\lambda^{h_{n}}-\lambda_{h_{n}},u^{h_{n}}-u_{h_{n}})\|_X
\lesssim \eta_a(V_{h_n})\delta_{h_{n}}(u).
 \end{eqnarray*}
This completes the proof.
\end{proof}

\begin{theorem}\label{event}
For Algorithm \ref{mg}, under the conditions of Thoerem \ref{mg estimate}, we have
\begin{eqnarray}
\|(\lambda-\lambda^{h_n},u-u^{h_n})\|_X &\lesssim& \delta_{h_n}(u),\\
\|(\lambda-\lambda^{h_n},u-u^{h_n})\|_0 &\lesssim& \eta_a(V_{h_n})\delta_{h_n}(u).
\end{eqnarray}
\end{theorem}
\begin{proof}
Theorem \ref{event} can be derived from Lemma \ref{lemma:Maday},
 Theorem \ref{mg estimate} and triangle inequality.
\end{proof}


\section{Work estimate of multigrid algorithm}
In this section, the computational work of Algorithm \ref{mg} is presented to
show the efficiency of this multigrid scheme.
Denote the dimension of finite element space $X_{h_k}$ by $N_k$. Then we have
$$N_k \approx \beta^{d(k-n)}N_n,\ \ \ k=1,2,\cdots,n.$$

\begin{theorem}
Assume the work of GPE problem in the initial finite element space $X_{h_1}$
is $\cO(M_{h_1})$ and that of the linear boundary value
problem in each level $X_{h_k}$ is $\cO(N_k)$ for $k=1,2,\cdots,n$.
Then the work involved in the multigrid method is $\cO(N_n+M_{h_1})$.
Furthermore, the complexity can be $\cO(N_n)$ provided $M_{h_1}\leq N_n$.
\end{theorem}
\begin{proof}
Denote the work in the $k-th$ finite element space $X_{h_k}$ by $W_k$ and the total work by $W$.
Then
\begin{eqnarray*}
W &=& \sum_{k=1}^{n}W_k = \cO(M_{h_1}+\sum_{k=2}^{n}N_k) \\
&=& \cO( M_{h_1}+\sum_{k=2}^{n}\beta^{d(k-n)}N_n ) \\
&=& \cO(M_{h_1}+N_n).
\end{eqnarray*}
Then we derive the desired result and  $W=\cO(N_n)$ when $M_{h_1}\leq N_n$.
\end{proof}

\section{Numerical results}
In this section, two numerical examples are presented to illustrate the
efficiency of the multigrid scheme proposed in this
paper.

\subsection{Example 1}
In the first example, we use  Algorithm \ref{mg} to solve the following GPE:
Find $(\lambda,u)$ such that
\begin{equation}\label{nonlinear_pde}
\left\{
\begin{array}{rcl}
-\triangle u+Wu +\zeta|u|^2 u&=&\lambda u, \quad {\rm in} \  \Omega,\\
u&=&0, \ \  \quad {\rm on}\  \partial\Omega,\\
\int_{\Omega}u^2d\Omega&=&1,
\end{array}
\right.
\end{equation}
where $\Omega$ denotes the three dimensional domain $(0,1)^3$, $ \zeta=1$
and $W=x_1^2+x_2^2+x_3^2$.

The sequence of finite element spaces are constructed by linear element on a series of meshes
produced by regular refinement with $\beta=2$ (producing $\beta^3$ congruent subelements).
Since the exact solution is not known, an adequate accurate approximation is choosen as the exact solution
to investigate the convergence behavior. 
The optimal error estimates can be obtained from the numerical results
which are presented in Figure \ref{error of bec1}.

In order to show the efficiency of Algorithm \ref{mg},
we also provide the running time of Algorithm \ref{mg}.
Here, all schemes are running on the same machine PowerEdge R720 with the linux system hereafter.
The corresponding results are presented in Table \ref{table1} and Figure \ref{error of bec1},
which imply the efficiency and linear complexity of Algorithm \ref{mg}.
\begin{figure}[ht]
\centering
\includegraphics[width=7cm]{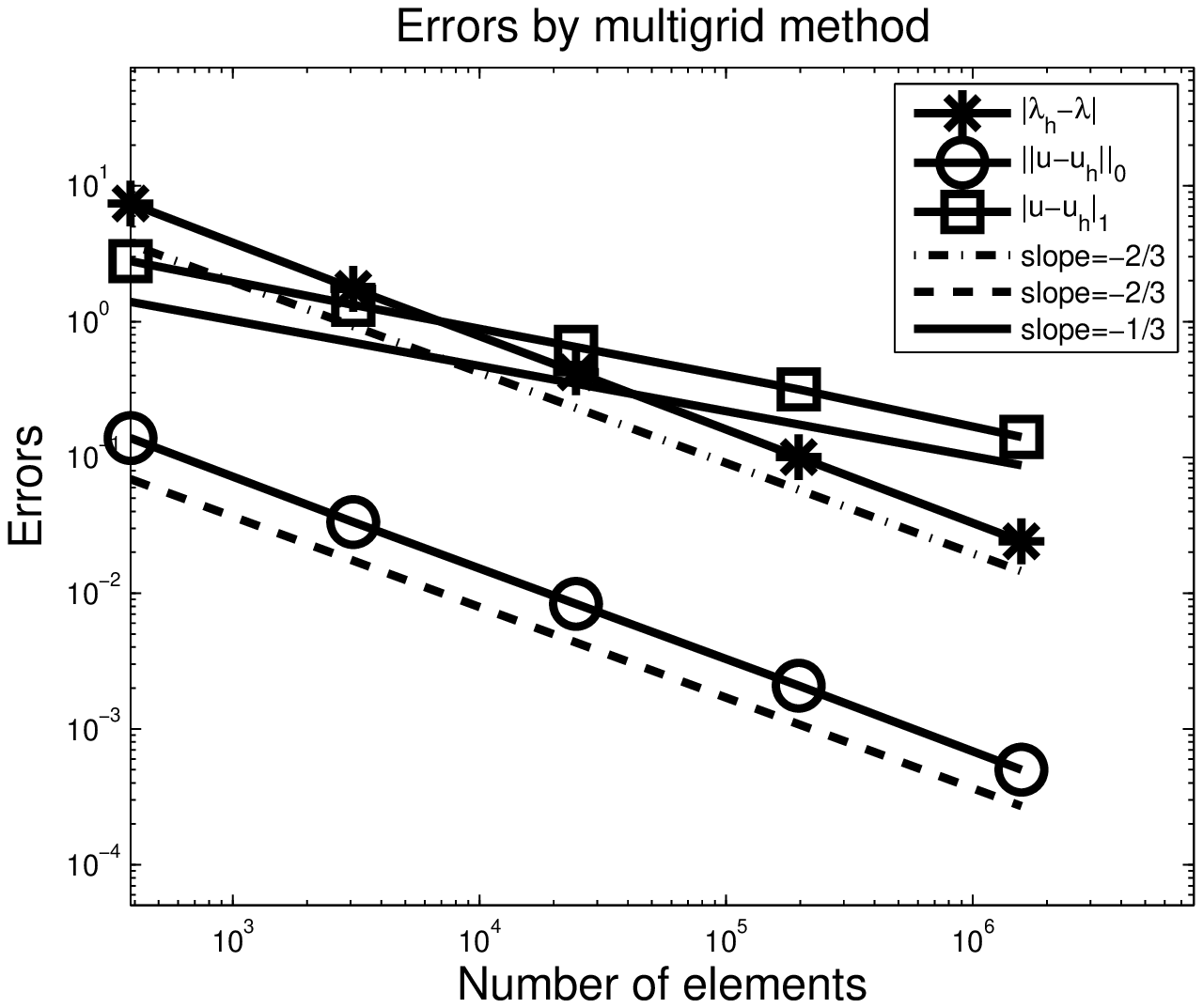}
\includegraphics[width=7cm]{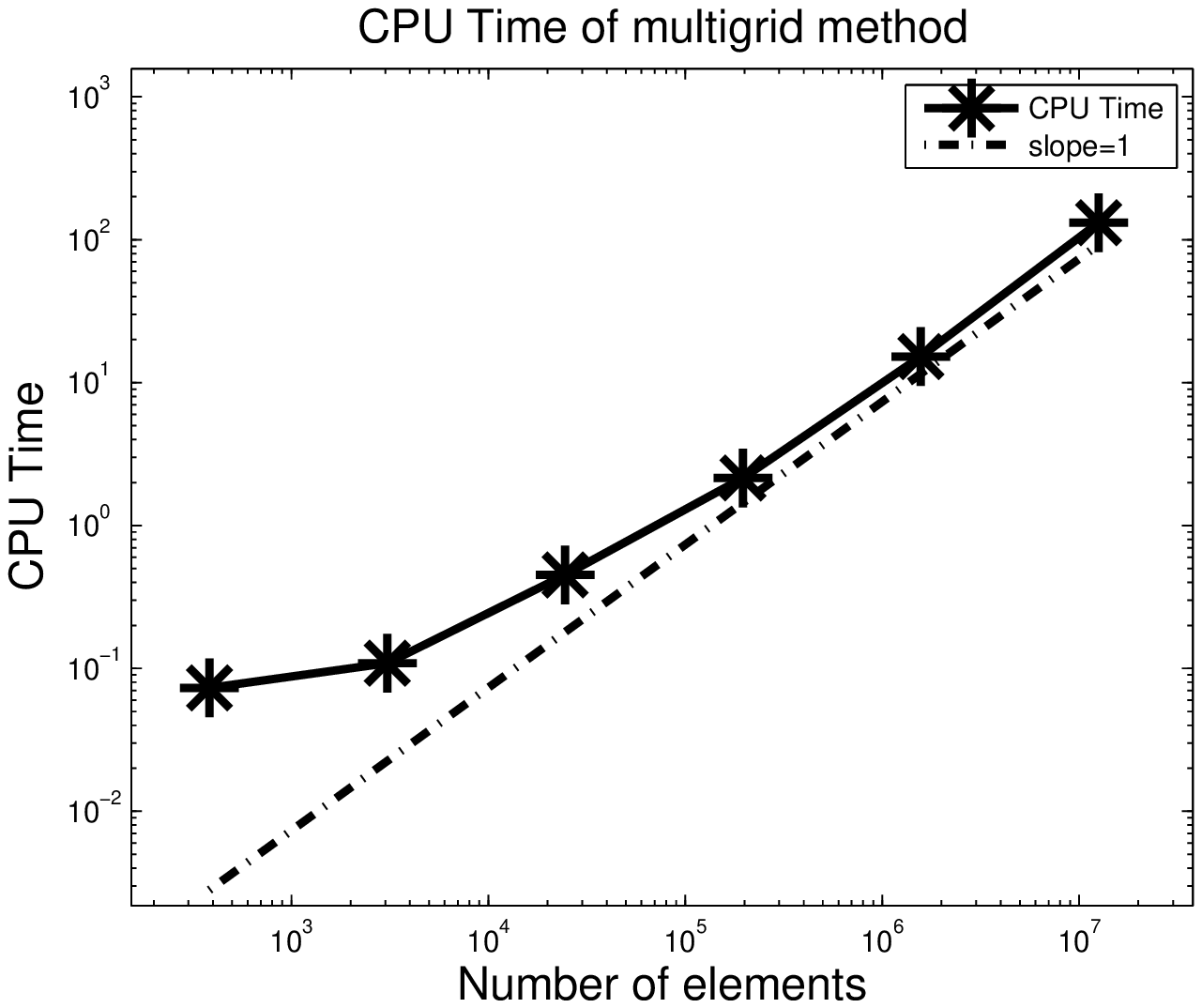}
\caption{\footnotesize \texttt Left: The errors of the multigrid method for the ground state solution of GPE,
where $\lambda_h$ and $u_h$ denote the numerical solutions of Algorithm \ref{mg}.
Right: CPU Time of Algorithm \ref{mg} for Example 1.}\label{error of bec1}
\end{figure}
\begin{table}[ht]
\begin{center}
\caption{\footnotesize The CPU time of Algorithm \ref{mg} for Example 1.}
\label{table1}
\begin{tabular}{|c|c|c|}\hline
Number of levels  & Number of elements & Time for Algorithm \ref{mg}\\
\hline
1 & 3072          & 0.1089             \\ \hline
2 & 24576         & 0.5249             \\ \hline
3 & 196608        & 2.1467             \\ \hline
4 & 1572846       & 15.7916            \\ \hline
5 & 12582912      & 131.3590           \\ \hline
\end{tabular}
\end{center}
\end{table}

\subsection{Example 2}
 In the second example, we consider
the GPE (\ref{nonlinear_pde}) on the domain $\Omega = (0,1)^3$ with the
coefficient $\zeta=100$ and $W=x_1^2+2x_2^2+4x_3^2$.
Numerical results are presented in Table \ref{table2} and Figure \ref{error of bec100}.
Hence the efficiency and linear complexity of Algorithm \ref{mg} can also be validated.
\begin{figure}[ht]
\centering
\includegraphics[width=7cm]{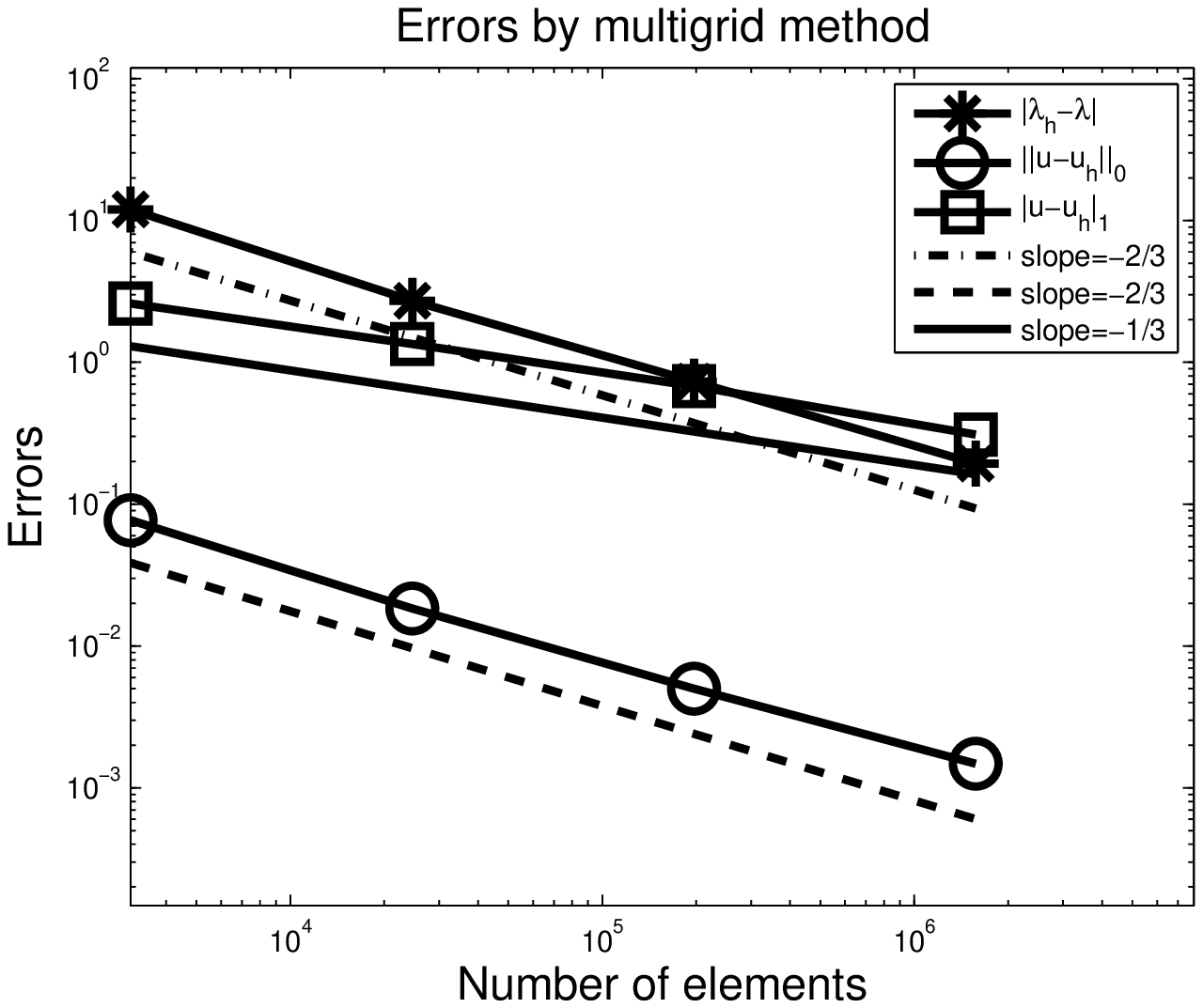}
\includegraphics[width=7cm]{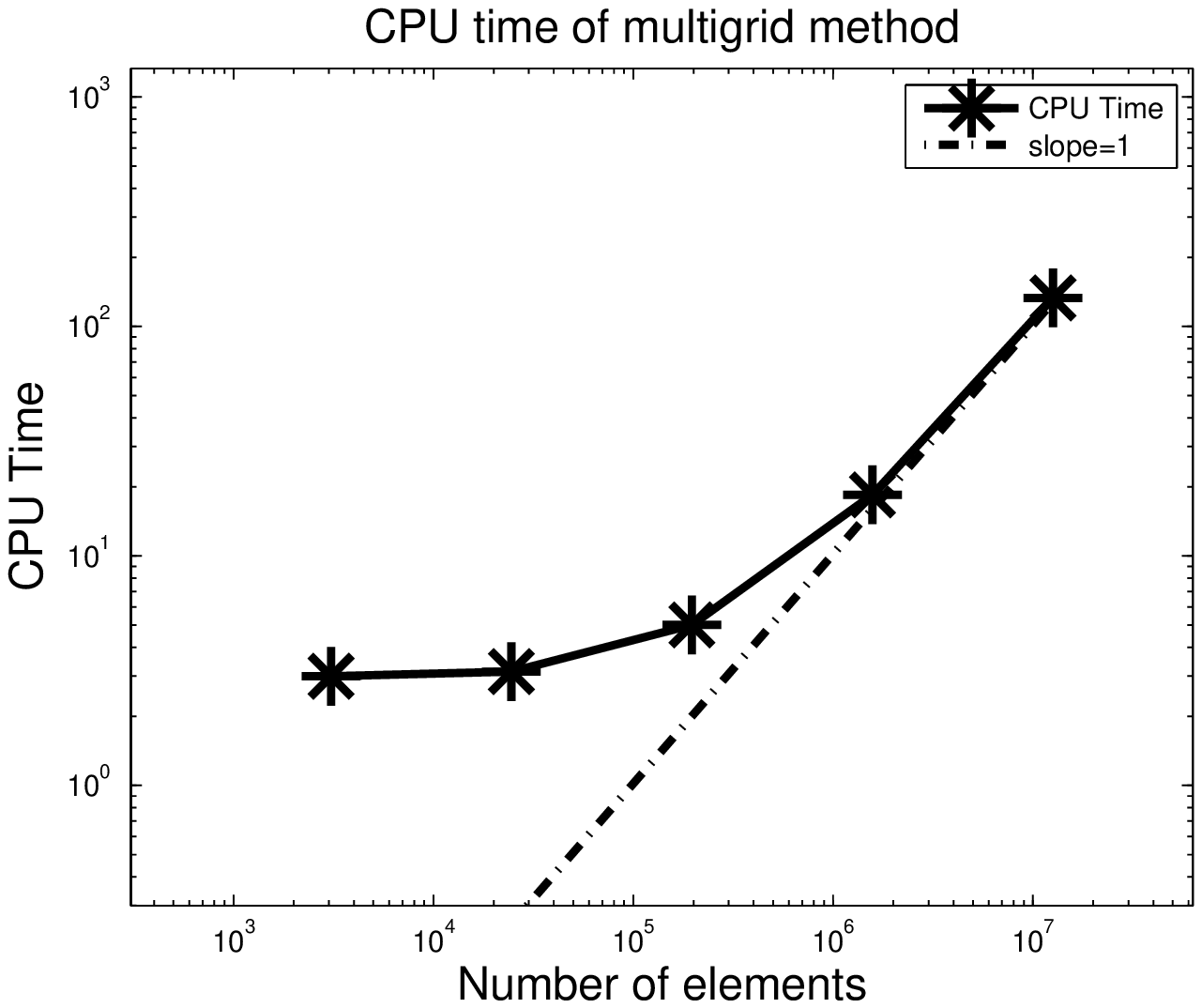}
\caption{\footnotesize Left: The errors of the multigrid method for the ground state solution of GPE,
where $\lambda_h$ and $u_h$ denote the numerical solutions of Algorithm \ref{mg}.
Right: CPU Time of Algorithm \ref{mg} for Example 2.}\label{error of bec100}
\end{figure}
\begin{table}[htp]
\begin{center}
\caption{\footnotesize The CPU time of Algorithm \ref{mg} for Example 2.}
\label{table2}
\begin{tabular}{|c|c|c|}\hline
Number of levels  & Number of elements & Time of Algorithm \ref{mg} \\ \hline
1 & 24576         & 2.4686              \\ \hline
2 & 196608        & 4.8973              \\ \hline
3 & 1572846       & 18.3522             \\ \hline
4 & 12582912      & 138.0450            \\ \hline
\end{tabular}
\end{center}
\end{table}

\section{Concluding remarks}
In this paper, we propose a type of multigrid method for GPE problems based on
Newton iteration. Different from the classical finite element method for
GPE problems, the proposed method transforms the nonlinear eigenvalue problem solving
to a series of linear boundary value problems solving and a eigenvalue
problem solving in the coarsest finite element space.
The high efficiency of linear boundary value problems solving can improve
 the overall efficiency of the simulation for BEC.
The corresponding analysis about the computational complexity has also been given.
The idea proposed here can also be extend to other nonlinear eigenvalue problems,
i.e., Kohn-Sham equation, which always arises from the electronic structure computation.



\begin{thebibliography}{39}
\bibitem{Adams}  R. A. Adams, {\em Sobolev spaces}, Academic Press, New York, 1975.



\bibitem{Adhikari}
S. K. Adhikari, {\em Collapse of attractive Bose-Einstein condensed
vortex states in a cylindrical trap}, Phys. Rev. E, 65 (2002), 016703.

\bibitem{AdhikariMuruganandam}
S. K. Adhikari and P. Muruganandam, {\em Bose-Einstein condensation dynamics from
the numerical solution of the Gross-Pitaevskii equation}, J. Phys. B, 35 (2002), 2831-2843.

\bibitem{AnderEnsherMattewWieman}
M. H. Anderson, J. R. Ensher, M. R. Mattews, C. E. Wieman and E. A. Cornell,
{\em Observation of Bose-Einstein
condensation in a dilute atomic vapor}, Science, 269 (1995), 198-201.

\bibitem{AnglinKetterle}
J. R. Anglin and W. Ketterle, {\em Bose-Einstein condensation of atomic gasses},
Nature, 416 (2002), 211-218.


\bibitem{Fortin}
F. Brezzi and F. Fortin, {\em  Mixed and Hybrid Finite Element Methods}, New York: Springer-Verlag, 1991.

\bibitem{BramblePasciak}
J. H. Bramble and J. E. Pasciak, {\em New convergence estimates for multigrid algorithms},
Math. Comp., 49 (1987), 311-329.

\bibitem{BrennerScott}
S. Brenner and L. Scott, {\em The Mathematical Theory of Finite Element
Methods}, New York: Springer-Verlag, 1994.

\bibitem{Bose}
S. N. Bose, {\em Plancks gesetz und lichtquantenhypothese}, Zeitschrift f\"{u}r Physik, 3(1924), 178-181.

{\bibitem{BaoCai}
W. Bao and Y. Cai, {\em Mathematical theory and numerical methods for
Bose-Einstein condensation}, Kinet. Relat. Models, 6(1) (2013), 1-135}.

{\bibitem{BaoDu}
W. Bao and Q. Du, {\em Computing the ground state solution of Bose--Einstein
condensates by a normalized gradient flow}, SIAM J. Sci. Comput., 25(5) (2004), 1674-1697.}
\bibitem{BaoTang}

W. Bao and W. Tang, {\em Ground-state solution of trapped interacting
Bose-Einstein condensate by directly minimizing the energy functional},
J. Comput. Phys., 187(2003), 230-254.

\bibitem{Cornell}
E. A. Cornell, {\em Very cold indeed: the nanokelvin physics of
Bose-Einstein condensation}, J. Res. Natl Inst. Stand.,
101 (1996), 419-434.

\bibitem{CornellWieman}
E. A. Cornell and C. E. Wieman, {\em Nobel Lecture: Bose-Einstein condensation
in a dilute gas, the first 70 years and some recent experiments},
Rev. Mod. Phys., 74 (2002), 875-893.

\bibitem{CancesChakirMaday}
E. Canc\`{e}s, R. Chakir and Y. Maday, {\em Numerical analysis of nonlinear eigenvalue problems},
J. Sci. Comput., 45(1-3) (2010), 90-117.

\bibitem{}  H. Chen, L. He and A. Zhou, {\em Finite element approximations
 of nonlinear eigenvalue problems in quantum physics},
Comput. Methods Appl. Mech. Engrg., 200(21) (2011), 1846-1865.

\bibitem{Ciarlet}
P. G. Ciarlet, {\em The Finite Element Method for Elliptic Problems},
Amsterdam: North-Holland, 1978.

\bibitem{DalGioPitaString}
F. Dalfovo, S. Giorgini, L. P. Pitaevskii and S. Stringari,
{\em Theory of Bose-Einstein condensation in trapped
gases}, Rev. Mod. Phys., 71 (1999), 463-512.

\bibitem{Einstein1}
A. Einstein, {\em Quantentheorie des einatomigen idealen gases}, Sitzungsberichte der Preussis-chen
           Akademie der Wissenschaften, 22 (1924), 261-267.

\bibitem{Einstein2}
A. Einstein, {\em Quantentheorie des einatomigen idealen gases, zweite abhandlung}, Sitzungs-berichte
           der Preussischen Akademie der Wissenschaften, 1 (1925), 3-14.

\bibitem{G}
 E. P. Gross, {\em Structure of a quantized vortex in boson systems}, Nuovo. Cimento., 20 (1961), 454-457.

\bibitem{LaudauLifschitz}
L. Laudau and E. Lifschitz, {\em Quantum Mechanics: non-relativistic theory},
Pergamon Press, New York, 1977.

\bibitem{LinXie}
Q. Lin and H. Xie, {\em A multi-level correction scheme for eigenvalue problems},
Math. Comp., 84(291) (2015), 71-88.

\bibitem{P}
L. P. Pitaevskii, {\em  Vortex lines in an imperfect Bose gas}, Soviet Phys. JETP, 13 (1961), 451-454.

\bibitem{PitaevskiiStringari}
L. P.  Pitaevskii and S. Stringari, {\em Bose-einstein condensation, Oxford University Press}, 2003.

\bibitem{Xie_JCP}
H. Xie, {\em A multigrid method for eigenvalue problem}, J. Comput. Phys., 274 (2014), 550-561.

\bibitem{Xu}
J. Xu, {\em Iterative methods by space decomposition and subspace
correction}, SIAM Review, 34(4) (1992), 581-613.

\bibitem{ZhouBEC}
A. Zhou, {\em An analysis of fnite-dimensional approximations
for the ground state solution of Bose-Einstein condensates},
Nonlinearity, 17 (2004), 541-550.

\bibitem{ZienkiewiczZhu}
O. Zienkiewicz and J. Zhu, {\em The superconvergent patch recovery
and a posteriori error estimates. Part 1: The recovery technique},
Internat. J. Numer. Methods Engrg., 33(7) (1992), 1331-1364.
\end{thebibliography}
\end{document}